\documentclass[pdftex]{amsart}
\usepackage{graphicx}
 \newtheorem{thm}{Theorem}[section]

 \theoremstyle{definition}
 \newtheorem{defn}[thm]{Definition}
 \theoremstyle{remark}
 \newtheorem{rem}[thm]{Remark}
 \newtheorem{ex}[thm]{Example}
 \numberwithin{equation}{section}

\begin{document}

%
%
%

\title{On the push-out spaces}

\author[M. Fathy]{M. Fathy}

\address{%
University of Applied Science\\
 and Technology of East Azarbayjan
Cooperation,\\ Tabriz, Iran}

\email{mortaza.fathy@yahoo.com}

\author{M. Faghfouri}
\address{University of Tabriz,\\ Tabriz, Iran}

 \email{faghfouri@tabrizu.ac.ir}



\begin{abstract}
Let
 $f:M^m\longrightarrow \Bbb R^{m+k}$  be an immersion where $M$ is a smooth connected
 $m$-dimensional manifold without boundary. Then we construct a subspace
 $\Omega(f)$ of $ \mathbb{R}^k$, namely push-out space. which corresponds to a set of embedded manifolds which are either parallel  to $ f $, tubes around $ f $ or, ingeneral, partial tubes around $ f $. This space  is invariant under
 the action of the normal holonomy group, $\mathcal{H}ol(f)$. Moreover, we
 construct geometrically
 some  examples for normal holonomy group and push-out space in ${\Bbb R}^3$.
 These examples will show that properties of push-out space that are proved in the case $\mathcal{H}ol(f)$ is trivial, is not true in general.
\end{abstract}

\maketitle
\section{Introduction}

In this paper we introduce push-out space  for an immersion
 $f:M^m\longrightarrow \Bbb R^{m+k}$, where $M$ is a smooth connected
 $m$-dimensional manifold without boundary. To do this, we give  some examples
in 3-dimensional Euclidean space, $\Bbb R^3$, infact, in these
examples we calculate normal holonomy group and push-out space
geometrically. We consider the case when $\mathcal{H}ol(f)$ is
non-trivial. This extends the work of Carter and Senturk \cite{2},
who obtained results about the case when $\mathcal{H}ol(f)$ is
trivial. In these examples we show that some of the properties of
push-out space which they obtained is not true for the case when
$\mathcal{H}ol(f)$ is non-trivial.

\section{Basic definitions}

\begin{defn}[\cite{2}]\label{de1} Let $f:M^m\longrightarrow \Bbb R^{m+k}$ be a
smooth immersion where $M$ is a smooth connected $m$-dimensional
manifold without boundary. The total space of the normal bundle of
$f$ is defined by
  $$N(f)=\{(p,x)\in M\times\Bbb R^{m+k}:<x,v>=0\quad  \forall v\in
  f_{*}T_{p}(M)\}$$
The endpoint map $\eta:N(f)\longrightarrow\Bbb R^{m+k}$ is defined
by $\eta(p,x)=f(p)+x$ and, the set of singular points of $\eta$ is
subset $\Sigma(f)\subset N(f)$ called the set of critical normals of
$f$ and the set of focal points of $\eta$ is a subset
$\eta({\Sigma(f)})\subset \Bbb R^{m+k}$.\\
For $p\in M$, we put $N_p(f)=\{x:(p,x)\in N(f)\}$ and
${\Sigma}_{p}(f)=\{x:(p,x)\in \Sigma(f)\}$ respectively, normal
space at $p$ and the set can be thought of as focal points with base
$p$.
\end{defn}
\begin{defn}[\cite{1}]\label{de2}
For $p_0\in M$ and  $p\in M$ and path $\gamma:[0,1]\longrightarrow
M$ from $p_0$ to $p$ define
$\varphi_{p,\gamma}:N_{p_0}(f)\longrightarrow N_p(f)$ by parallel
transport along $\gamma$. The $\varphi_{p,\gamma}$'s are isometries.
The normal holonomy group on $N_{p_0}(f)$, is
$$\mathcal{H}ol(f)=\{\varphi_{p_0,\gamma}:\gamma:[0,1]\longrightarrow M,\quad \gamma(0)=\gamma(1)=p_0\}$$
If the closed path $\gamma$ at $p_0$ is homotopically trivial then
$\varphi_{p_0,\gamma}$ is an element of the restricted normal
holonomy group ${\mathcal{H}ol}_0(f)$.
\end{defn}
\begin{defn}[\cite{3}]\label{de3}
For a fix $p_0\in M$ the push-out space for an immersion
$f:M^m\longrightarrow \Bbb R^{m+k}$ is defined by
$$\Omega(f)=\{x\in N_{p_0}(f):\forall p\in M,\forall \gamma\mbox{ s.t. }\gamma(0)=p_0,\gamma(1)=p \mbox{ then } \varphi_{p,\gamma}x\notin\Sigma_p(f) \}$$
(i.e. $\forall p\in M$, $f(p)+\varphi_{p,\gamma}(x)$ is not a focal
point with base $p$ when $x$ belongs to $\Omega(f)$). Therefore
$\Omega(f)$ is the set of normals at $p_0$, where transported
parallely along all curves, do not meet focal points. So $\Omega(f)$
is invariant under the action of $\mathcal{H}ol(f)$.
\end{defn}

\begin{defn}[\cite{4}]\label{de4}
Let $B\subset N(f)$ be a smooth subbundle with type fiber S where\\
1) S is a smooth submanifold of $\Bbb R^k$\\
2) $B\cap\Sigma(f)=\emptyset$\\
3) B is invariant under parallel transport (along any curve in M).
Then B is a smooth manifold and $g\equiv\eta|_B:B\longrightarrow\Bbb
R^{m+k}$ is a smooth immersion called a partial tube about f.
\end{defn}

\begin{thm}[\cite{2}]\label{2.5}
Let $ \mathcal{H}ol(f) $ is trivial and $ M $ be a compact manifold. Then each path-connected component of
$\Omega(f)$ is open in $\Bbb R^k$.
\end{thm}

\begin{thm}[\cite{2}]\label{2.6}
Let $ \mathcal{H}ol(f) $ is trivial then
Each path-connected component of $\Omega(f)$ is convex.
\end{thm}

\begin{rem}
In Example \ref{ex:3.2}, if $\frac{\alpha}{\pi}$ is irrational then
$\Omega(\bar{f})$ is not open in $\Bbb R^2$ but $M=\Bbb S^1$ is
compact. Also, in Example \ref{ex:3.5}, $\Omega({f})=\{O\}$ hence
$\Omega({f})$ is closed in $\Bbb R^2$ but $\mathcal{H}ol(f)$ is
trivial. This shows that Theorem \ref{2.5} is false when M is not
compact or $\mathcal{H}ol(f)$ is non-trivial.
\end{rem}

\begin{rem}
In Example \ref{ex:3.6} one of path-connected components of
$\Omega(f)$, which is the complement space of cone and two other
components in $\Bbb R^3$, is not convex. This shows that Theorem
\ref{2.6} is false when $\mathcal{H}ol(f)$ is non-trivial.
\end{rem}
we conclude that the properties of push-out space that are proved in
the case $\mathcal{H}ol(f)$ is trivial, is not true in general.

\section{Examples of normal holonomy groups and push-out spaces}

\begin{ex}\label{ex.31}
We start with a curve as below
 \begin{center}
  \includegraphics[width=4cm]{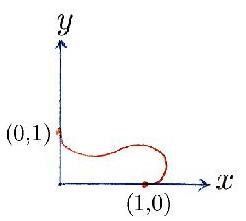}
\end{center}
 suppose this curve is given by $s\mapsto
(\xi(s),\eta(s))$ where $s\in [0,1]$ and at $(1,0,0)$:$s=0$
,${\frac{\partial\xi} {\partial s}}=1,({\frac{\partial}{\partial
s}})^{r}\eta=0$ for all $r\geq 0$ and at
$(0,1,0)$:$s=1$,${\frac{\partial\eta} {\partial
s}}=1,({\frac{\partial} {\partial s}})^{r}\xi=0$ for all $r\geq 0$.
Now,we take this curve in $\Bbb R^3$ and consider the same curves in
$yz$-plane and $xz$-plane and fit together to make a smooth closed
curve in $\Bbb R^3$.
Now by identifying $\Bbb S^1$ with $\frac {\Bbb R} {3\Bbb Z}$, the
curve in $\Bbb R^3$  can be redefined as $f:\Bbb
S^1\longrightarrow\Bbb R^3$ where:

$f(s)= \left\{
  \begin{array}{cc}
     (\xi(s),\eta(s),0) \qquad\quad \quad 0\leq s\leq 1 \\
     (0,\xi(s-1),\eta(s-1))  \quad 1\leq s\leq 2\\
     (\eta(s-2),0,\xi(s-2)) \quad 2\leq s\leq3\\
  \end{array}
\right.$

\begin{center}
  \includegraphics[width=5cm]{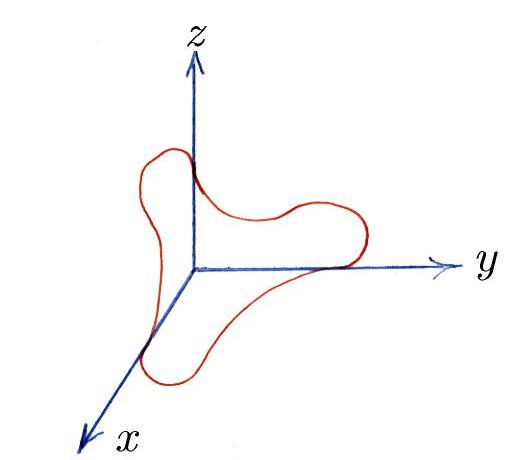}
\end{center}
\end{ex}
To find the normal holonomy group of the above curve, we will
consider normal vector to the curve under parallel transport. As
each part of the curve lies in a 2-plane, the normal plane at a
point of the curve is spanned by the perpendicular direction to the
2-planes.\\
{\bf Step1.}Start with the normal vector at (1,0,0), in the diagram,
it stays in the xy-plane under parallel transport.
\begin{center}
 \includegraphics[width=6cm]{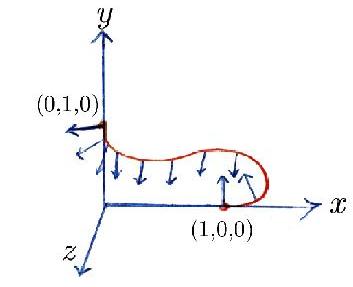}
\end{center}
The normal vector (0,1,0) at (1,0,0) goes to normal vector (-1,0,0)at (0,1,0).\\
 {\bf Step2.}At (0,1,0)the the normal vector (-1,0,0) is
 perpendicular to the yz-plane, it stays perpendicular to the
 yz-plane under parallel transport form (0,1,0) to (0,0,1).
 \begin{center}
 \includegraphics[width=6cm]{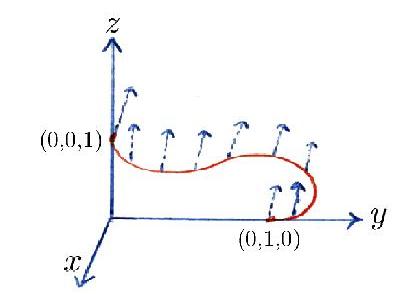}
\end{center}
The normal vector (-1,0,0)at(0,1,0) goes to normal vector(-1,0,0) at (0,0,1).\\
{\bf Step3.}The the normal vector (-1,0,0) is in the xz-plane at
(0,0,1) and stays in the xz-plane from (0,0,1) to (1,0,0).
\begin{center}
\includegraphics[width=6cm]{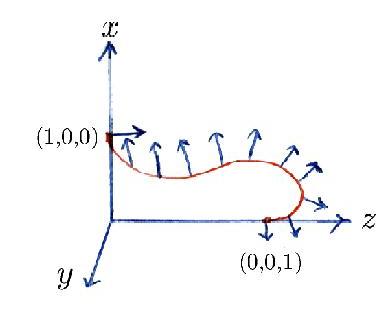}
\end{center}
The normal vector (-1,0,0,) at (1,0,0)  by going once around the
curve the normal vector will turn about $\frac {\pi}{2}$.
\begin{center}
 \includegraphics[width=6cm]{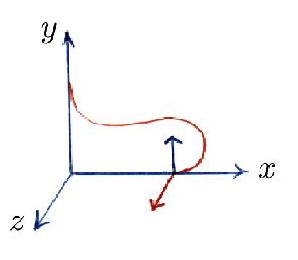}
\end{center}
Going around of curve again, the normal vector moves through another
$\frac {\pi}{2}$ and after four times around the curve back to its
original position. this shows that $\mathcal{H}ol(f)$ is generated
by a rotation through $\frac {\pi}{2}$.
\begin{center}
\includegraphics[width=6cm]{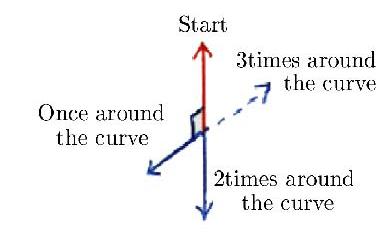}
\end{center}
Now we find the push-out space of $f$. Except at end-points of three
areas, locally the curve lies in a 2-plane so the focal points with
base $s$,$f(s)+\Sigma_s(f)$, consists of a straight line through the
center of curvature, $c(s)$, of the curve at $s$, perpendicular to
the line joining $c(s)$ and $f(s)$.
\begin{center}
\includegraphics[width=6cm]{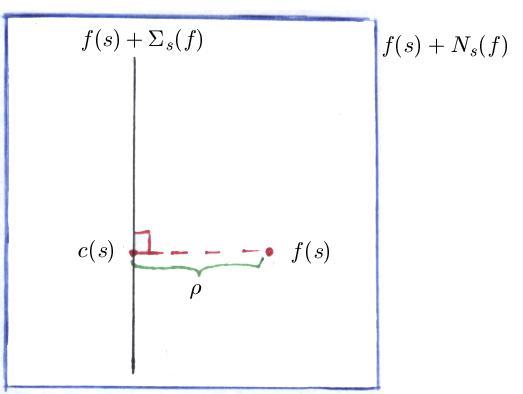}
\end{center}
At end-points of three areas, and possibly some other points, the
focal set is empty as the center of curvature "at infinity".\\
 so
$\Sigma_s(f)$ is a line in $N_s(f)$. The image of $\Sigma_s(f)$
under normal holonomy group is obtained by rotating it through
$\frac {\pi}{2}$ until it returns to the original position.
\begin{center}
 \includegraphics[width=6cm]{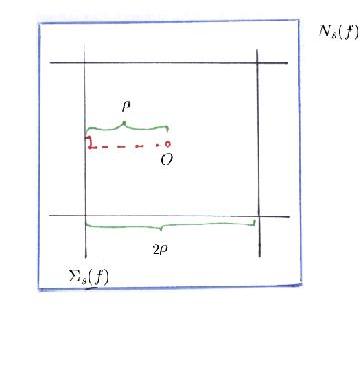}
\end{center}
Now, fix the normal plane $N_{s_0}(f)$ at $f(S_0)=(1,0,0)$ where
$s_0=0$ and use parallel transport to identify all the normal planes
with the normal plane  $N_{s_0}(f)$. The push-out space is
complement of all the $\Sigma_s(f)$ and their images under normal
holonomy group.
\begin{center}
 \includegraphics[width=6cm]{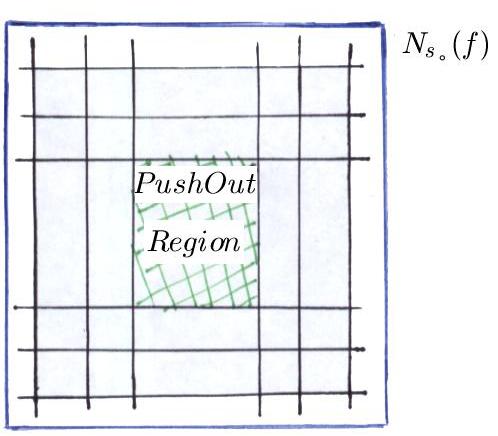}
\end{center}
Therefore the push-out space of $f:\Bbb S^1\longrightarrow\Bbb R^3$
is an open square $Q$ with sides of length $2\rho$ where $\rho$ is
the minimum absolute value of the radius of curvature of the
original curve in the xy-plane. (i.e. $\Omega(f)$)is the interior of
the smallest square on $N_{s_0}(f)$.)

\begin{ex}\label{ex:3.2}
We consider the immersion $\bar{f}$ as in Example \ref{ex.31} except that
the xz-plane is tilted through an angle $\alpha$.

In other words, $\bar{f}=Lof$ where f is the immersion in example \ref{ex.31} and $L$ is the linear
transformation given by\\
$L=\left(
    \begin{array}{ccc}
    1 & 0 & 0 \\
    0 & 1 & \tan\alpha \\
    0 & 0 & 1 \\
    \end{array}
  \right)$ where
$0<\alpha<\frac {\pi} {2}$.
So, we have
 $$\bar{f}(s)=\begin{cases}
     (\xi(s),\eta(s),0)& 0\leq s\leq 1 \\
     (0,\xi(s-1)+\eta(s-1)\tan{\alpha},\eta(s-1))  &1\leq s\leq 2\\
     (\eta(s-2),\xi(s-2)\tan{\alpha},\xi(s-2)) & 2\leq s\leq3
\end{cases}$$

 The end-points of three areas of this immersion are
(1,0,0),(0,1,0) and (0,$\tan\alpha$,1). Note that at these points
the tangent to the curve is the radial line form (0,0,0) so unit
tangent at (1,0,0) is (1,0,0)and unit tangent at (0,$\tan\alpha$,1)
is
$\frac{(0,\tan\alpha,1)}{\sqrt{1+\tan^{2}\alpha}}$ etc.\\
As in Example \ref{ex.31}, under parallel transport, the normal vector
(0,1,0) at (1,0,0) goes to the normal vector (-1,0,0) at (0,1,0),
which goes to the normal vector (-1,0,0) at (0,$\tan\alpha$,1),
which goes to the normal
vector$\frac{(0,\tan\alpha,1)}{\sqrt{1+\tan^{2}\alpha}}$ at (1,0,0).
So going once around the curve the normal vector has moved through
$\frac {\pi}{2}-\alpha$.
\begin{center}
  \includegraphics[width=6cm]{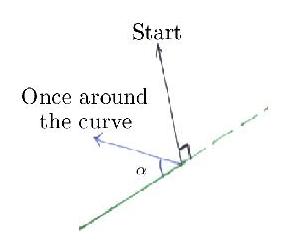}
\end{center}

 This shows that  $\mathcal{H}ol(\bar{f})$
is generated by a rotation through $\frac
  {\pi}{2}-\alpha$
As in Example 3.1, the image of $\Sigma_s(\bar {f})$ under normal
holonomy group is obtained by rotating the line $\Sigma_s(\bar {f})$
through $\frac {\pi}{2}-\alpha$. It depends on $\alpha$ and is
obtained as:
\begin{center}
\includegraphics[width=6cm]{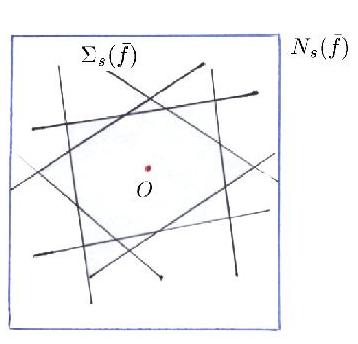}
\end{center}
If $R$ is the rotating through an angle $\frac {\pi}{2}-\alpha$ then
 $\Omega({\bar{f}})=\bigcap\{(R)^nQ: n\in\Bbb Z\}$ where Q is a
square as in Example \ref{ex.31}, thus if $\frac {\alpha}{\pi}$ is rational
then $(R)^n(\Sigma_s(\bar {f}))=\Sigma_s(\bar {f})$ for some
$n\in\Bbb Z $ and so,$\Omega({\bar{f}})$  is the interior of the
smallest polygon. If $\frac {\alpha}{\pi}$ is irrational then
$(R)^n(\Sigma_s(\bar {f}))\neq \Sigma_s(\bar {f})$ for any $n\in\Bbb
Z $ and so,$\Omega({\bar{f}})$  is an open disk of radius $\rho$
together with a dense set of points on the boundary circle where
$\rho$ is minimum absolute value of the radius of curvature of the
immersed curve by $\bar{f}$.
\end{ex}

\begin{ex}\label{ex:3.3} In Example \ref{ex:3.2}, we replace the immersion $\bar{f}$ with the
immersion $\bar{f}oh$ where $h:\Bbb R\longrightarrow\Bbb
S^1\equiv\frac {\Bbb R} {3\Bbb Z}$ is covering projection. Since
${\Bbb R}$ is simply connected, for any arbitrary point $s\in {\Bbb
R}$, any closed path at $s$ is nulhomotopic with constant path at
$s$, hence definition \ref{de2} shows that, the normal holonomy
group of $\bar{f}Oh$ is trivial
(i.e.$\mathcal{H}ol(\bar{f}oh)=\mathcal{H}ol_0(\bar{f}oh)$ ). To
calculate $\Omega(\bar{f}oh)$, we prove the theorem \ref{th:3.4}, in
general. It will show that $\Omega(\bar{f}oh)=\Omega(\bar{f})$.
\end{ex}

\begin{thm}\label{th:3.4}
Let $f:M^m\longrightarrow \Bbb R^{m+k}$ be an immersion and
$\hat{M}$ be any covering space with covering projection
$h:\hat{M}\longrightarrow M$. If $\hat{f}=foh$, then
$\Omega(\hat{f})=\Omega({f})$.
\end{thm}
\begin{proof}
 Let $x\in\Omega({f}) $ and fix $p_0\in M$. Then
definition \ref{de3} implies that, $\forall p\in M,\forall \gamma$ s.t.
$\gamma(0)=p_0, \gamma(1)=p$; $\varphi_{p,\gamma}x\neq \Sigma_p(f).$\\
we define the total space of the normal bundle of $\hat{f}$ by
$$N(\hat{f})=\{(\hat{p},x)\in \hat{M}\times\Bbb R^{m+k}:<x,v>=0\quad  \forall v\in
 \hat{ f}_{*}T_{\hat{p}}(\hat{M})\}$$
  Also, for any $\hat{p}\in h^{-1}(p)$
  we have
  \begin{align*}
\hat{f}_{*}T_{\hat{p}}(\hat{M})&=(foh)_{*}T_{\hat{p}}(\hat{M})\\
&=(f_{*}oh_{*})T_{\hat{p}}(\hat{M})\\
&=f_{*}T_p(M)
  \end{align*}
  this shows that, for any $\hat{p}\in h^{-1}(p)$ we have $N_{\hat{p}}(\hat{f})=N_p(f)$
  and so$\Sigma_{\hat{p}}(\hat{f})=\Sigma_p(f)$. Further, we fix$\hat{p_0}\in
  h^{-1}(p_0)$ then
  $\hat{\varphi}_{\hat{p},\hat{\gamma}}=\varphi_{p,\gamma}$ where
  $\hat{\gamma}:[0,1]\longrightarrow\hat{M}$
   s.t. $\hat{\gamma(0)}=\hat{p_0}, \hat\gamma(1)=\hat{p}.$
   Therefore,$\forall \hat{p}\in \hat {M}, \forall \hat{\gamma}$;
   $\hat{\varphi}_{\hat{p},\hat{\gamma}}x \neq \Sigma_{\hat{p}}(\hat{f}).$
Now using definition \ref{de3} again, follows that, $x\in
\Omega(\hat{f})$. By the same way proves that
$\Omega(\hat{f})\subseteq \Omega(f)$.
\end{proof}

\begin{ex}\label{ex:3.5}
If $\frac {\pi}{2}-\alpha=\frac {2\pi}{n}$, then Example \ref{ex:3.3} can be
modified by replacing h by the n-fold covering $\bar{h}:\Bbb
S^1\longrightarrow \Bbb S^1$.

Going once around the first $\Bbb S^1$ in $\bar{h}:\Bbb
S^1\longrightarrow \Bbb S^1$ corresponds to moving n times around
the second $\Bbb S^1$ so parallely transporting a normal n times
around the second $\Bbb S^1$ which gives a rotation of
\begin{align*}
n(\frac {\pi}{2}- \alpha)&=n(\frac{2\pi}{n})\\
&=2\pi
  \end{align*}
  i.e. the identity, so
$\mathcal{H}ol(\bar{f}o\bar{h})=\mathcal{H}ol_0(\bar{f}o\bar{h})$.
Since, the immersed curve by $\bar{f}o\bar{h}$ and the immersed
curve by $\bar{f}$ have same figure in $\Bbb R^3$ and
$\mathcal{H}ol(\bar{f}o\bar{h})$ is trivial so the singular sets of
them also the same (i.e. $\Sigma(\bar{f}o\bar{h})=\Sigma(\bar{f})$).
This implies that $\Omega(\bar{f}o\bar{h})=\Omega(\bar{f})$.
\end{ex}

\begin{ex}\label{ex:3.6}
We consider a sequence of curves $f_n$ in $\Bbb R^3$ defined as in
Example \ref{ex.31} except that $||f_n(s)||$ and the curvature tends to
infinity with n when $s=\frac {1}{2},\frac {3}{2}$ or $\frac {5}{2}$
but is bounded otherwise.\\
\begin{center}
 \includegraphics[width=6cm]{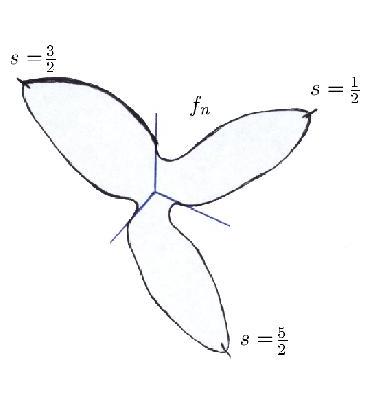}
\end{center}
Now, we define the immersion $f:\Bbb R\longrightarrow\Bbb R^3$ by
$f(s\pm3n)=f_n(s)$. When n tends to infinity, the immersion $f:\Bbb
R\longrightarrow\Bbb R^3$ has a sequence of points where the
curvature tends to infinity and the radius of curvature at these
points can be arbitrary small; in other words, $\exists s$ where
$\Sigma_s(f)$ is arbitrary close to "O" in $N_s(f)$. So $\{O\}$ is
the only point not in the image of $\Sigma_s(f)$ under normal
holonomy group for all $s\in \Bbb R.$ Then $\Omega(f)=\{O\}$. In
this case because $\Bbb R$ is simply connected then
$\mathcal{H}ol(f)$ is trivial.
\end{ex}

The following results  have been proved in \cite{2}, when
$\mathcal{H}ol(f)$ is trivial.
\begin{thm}\label{th:3.7}
Let
M be a compact manifold, then each path-connected component of
$\Omega(f)$ is open in $\Bbb R^k$.
\end{thm}

\begin{thm}\label{th:3.8}
Each path-connected component of $\Omega(f)$ is convex.
\end{thm}

\begin{rem}
In Example \ref{ex:3.2}, if $\frac{\alpha}{\pi}$ is irrational then
$\Omega(\bar{f})$ is not open in $\Bbb R^2$ but $M=\Bbb S^1$ is
compact. Also, in Example \ref{ex:3.5} $\Omega({f})=\{O\}$ so
$\Omega({f})$ is closed in $\Bbb R^2$ but $\mathcal{H}ol(f)$ is
trivial. This shows that Theorem \ref{th:3.7} is false when M is not
compact or $\mathcal{H}ol(f)$ is non-trivial.
\end{rem}

\begin{rem}
In Example \ref{ex:3.6} one of the path-connected components of
$\Omega(f)$, which is the complement space of cone and two other
components in $\Bbb R^3$, is not convex. This shows that Theorem
\ref{th:3.8} is false when $\mathcal{H}ol(f)$ is non-trivial.
\end{rem}
Thus the properties of push-out space that are proved in the case
$\mathcal{H}ol(f)$ is trivial, is not true in general.

\end{document}